\documentclass[11pt,reqno]{amsart}
\usepackage{graphicx}
\usepackage{float}
\usepackage{amsfonts}
\usepackage[caption = false]{subfig}
\usepackage{amsmath}
\usepackage{graphics}
\usepackage{float}
\usepackage{enumerate}
\usepackage{mathtools}
\usepackage{amsthm}
\usepackage[english]{babel}
\usepackage{amsfonts,amssymb,comment}
\usepackage{mathrsfs}
\usepackage[utf8x]{inputenc}\usepackage[english]{babel}
\usepackage{soul}
\usepackage[all]{xy,xypic}
\usepackage{cite}
\usepackage{relsize}
\numberwithin{equation}{section}
\newtheorem{theorem}{Theorem}[section]

\newtheorem{lemma}[theorem]{Lemma}

\theoremstyle{definition}
\newtheorem{definition}[theorem]{Definition}
\theoremstyle{remark}

\numberwithin{equation}{section}
\DeclareMathOperator{\RE}{Re}

\begin{document}
	\title[Second and Third order differential subordination for exponential function]{Second and Third order differential subordination for exponential function}

	\author[S. Sivaprasad Kumar]{S. Sivaprasad Kumar}
	\address{Department of Applied Mathematics, Delhi Technological University, Delhi--110042, India}
	\email{spkumar@dce.ac.in}

	\author[Neha Verma]{Neha Verma}
	\address{Department of Applied Mathematics, Delhi Technological University, Delhi--110042, India}
	\email{nehaverma1480@gmail.com}

	
	\maketitle
\begin{abstract}
This article presents several findings regarding second and third-order differential subordination of the form:
$$
    p(z)+\gamma_1 zp'(z)+\gamma_2 z^2p''(z)\prec h(z)\implies p(z)\prec e^z
$$
and $$
    p(z)+\gamma_1 zp'(z)+\gamma_2 z^2p''(z)+\gamma_3 z^3p'''(z)\prec h(z)\implies p(z)\prec e^z.
$$
Here, $\gamma_1$, $\gamma_2$, and $\gamma_3$ represent positive real numbers, and various selections of $h(z)$ are explored within the context of the class $\mathcal{S}^{*}_{e} := \{f \in \mathcal{A} : zf'(z)/f(z) \prec e^z\}$, which denotes the class of starlike functions associated with the exponential function.

\end{abstract}
\maketitle

\section{Introduction}
	\label{intro}
\noindent 
Let $\mathcal{A}$ represent the class of normalized analytic functions defined on the open unit disk $\mathbb{D}:=\{z\in \mathbb{C}:|z|<1\}$ of the form
$$\mathcal{A}=\{f\in \mathcal{H}:f(z)=z+a_2z^2+a_3z^3+\cdots\}$$
where $\mathcal{H}=\mathcal{H}(\mathbb{D})$ is the class of analytic functions defined on $\mathbb{D}$. For any positive integer $n$ and $a\in \mathbb{C}$, we define
\begin{equation*}
    \mathcal{H}[a,n]:=\{f\in \mathcal{H}:f(z)=a+a_nz^n+a_{n+1}z^{n+1}+a_{n+2}z^{n+2}+\cdots\}.
\end{equation*}
Suppose that $\mathcal{S}\subset\mathcal{A}$ is the collection of all univalent functions. A function $f\in \mathcal{S}$ is called starlike if and only if $\RE(zf'(z)/f(z))>0$ and the class consisting of all such functions is denoted by $\mathcal{S}^{*}$. If $h$ and $g$ belong to the class $\mathcal{A}$, we say $g$ is subordinate to $h$ (i.e., $g\prec h$) provided there exists a Schwarz function $w$ with $w(0)=0$ and $|w(z)|\leq |z|$ such that $g(z)=h(w(z))$. Additionally, if $h$ is univalent, then $g\prec h$ if and only if $h(0)=g(0)$ and $g(\mathbb{D})\subseteq h(\mathbb{D})$. A significant milestone in the advancement of differential subordination occurred with the publication of Miller and Mocanu's monograph \cite{miller}. This work initiated a substantial revolution in the research community of geometric function theory. In 1992, Ma and Minda \cite{ma-minda} introduced the class $\mathcal{S}^{*}(\varphi)$ through subordination unifying several subclasses of $\mathcal{S}^{*}$, given by
\begin{equation}
		\mathcal{S}^{*}(\varphi)=\bigg\{f\in \mathcal {A}:\dfrac{zf'(z)}{f(z)}\prec \varphi(z) \bigg\}\label{mindaclass}
\end{equation}
with $\varphi$ as an analytic univalent function so that $\RE\varphi(z)>0$, $\varphi(\mathbb{D})$ is symmetric about the real axis and starlike with respect to $\varphi(0)=1$ with $\varphi'(0)>0$. These functions are termed as Ma-Minda functions. By specializing the function $\varphi$ in \eqref{mindaclass}, researchers have derived various noteworthy subclasses of $\mathcal{S}^*$. Raina and Sok'{o}\l \cite{raina} introduced the class $\mathcal{S}^{*}_{q}$ by opting for $\varphi(z)=z+\sqrt{1+z^2}$, while Cho et al. \cite{chosine} studied the class of starlike functions $\mathcal{S}^{*}_{S}$, wherein $\varphi(z)=1+\sin z$. Kumar and Kamaljeet \cite{kumar-ganganiaCardioid-2021} and Arora and Kumar \cite{kush} explored the classes $\mathcal{S}^{*}_{\varrho}$ and $\mathcal{S}^{*}_{\rho}$, respectively, by selecting $\varphi(z)$ as $1+ze^z$ and $1+\sinh^{-1} z$.

Likewise, Mendiratta et al. \cite{mendi} introduced and examined the class $\mathcal{S}^{*}_{e}:=\mathcal{S}^{*}(e^z)$ by choosing $\varphi(z)=e^z$ in \eqref{mindaclass}. Geometrically, a function $f$ belongs to $\mathcal{S}^{*}_{e}$ if and only if $zf'(z)/f(z)\in \Delta_e:=\{\varpi\in \mathbb{C}:|\log \varpi|<1\}$. In their study, they derived conditions on the parameter $\beta$ such that $p(z)\prec e^z$ holds, given that $1+\beta zp'(z)/p(z)\prec h(z)$ for $h(z)=(1+Cz)/(1+Dz)$, $\sqrt{1+z}$ and $e^z$. Furthermore, Kumar and Ravichandran \cite{sushil} extended their findings and established sharp bounds on $\beta$ to ensure $p(z)\prec e^z$ in cases where $1+\beta z p'(z)/p^j(z)\prec h(z)$ for $j=0,2$ and $h(z)=(1+Cz)/(1+Dz)$ or $\sqrt{1+z}$. Subsequently, Naz et al. \cite{adibastarlikenessexponential} further extended these results by deriving a broad range of first-order subordination implications for the exponential function and generalizing the results of \cite{mendi}. To explore the latest developments in the theory of differential subordination and coefficient related problems for the exponential function, one can refer to \cite{goelhigher,ckms,mushtaq,nehaexpo,madaan,goel}.
Recently, Verma and Kumar \cite{nehadiffexpo} derived conditions on the parameters $\alpha$, $\beta$ and $\gamma$ while considering the expressions $1+\gamma zp'(z)+\beta z^2 p''(z)\prec h(z)$ and $1+\gamma zp'(z)+\beta z^2 p''(z)+\alpha z^3 p'''(z)\prec h(z)$ so that $p(z)\prec e^z$ for several choices of the function $h(z)$ such as $\sqrt{1+z}$, $2/(1+e^{-z})$ and $e^z$. Motivated by the aforementioned work, this paper concentrates on establishing implications of second and third-order differential subordination involving the exponential function. In the following, we present essential definitions and notations which are necessary to proceed further.
\begin{definition}
Let $Q$ denote the set of analytic and univalent functions $q\in\overline{\mathbb{D}}\setminus \mathbb{E}(q)$, where
\begin{equation*}
   \mathbb{E}(q)=\{\zeta\in \partial \mathbb{D}:\lim_{z\rightarrow\zeta}q(z)=\infty\}
\end{equation*}
such that $q'(\zeta)\neq 0$ for $\zeta\in \partial \mathbb{D}\setminus \mathbb{E}(q)$. The subclass of $Q$ for which $q(0)=a$ is denoted by $Q(a)$.
\end{definition}

\begin{lemma}\cite{antoninoandmiller}\label{lemmaformk}
Let $z_0\in \mathbb{D}$ and $r_0=|z_0|$. Let $f(z)=a_nz^n+a_{n+1}z^{n+1}+\cdots$ be continuous on $\overline{\mathbb{D}}_{r_0}$ and analytic on $\mathbb{D}\cup\{z_0\}$ with $f(z)\neq 0$ and $n\geq 2$. If $|f(z_0)|=\max \{|f(z)|:z\in \overline{\mathbb{D}}_{r_0}\}$ and $|f'(z_0)|=\max\{|f'(z)|:z\in \overline{\mathbb{D}}_{r_0}\}$, then there exist real constants $m$, $k$ and $l$ such that
\begin{equation*}
      \frac{z_0f'(z_0)}{f(z_0)}=m,\quad 1+\frac{z_0f''(z_0)}{f'(z_0)}=k\quad \text{and}\quad 2+\RE\bigg(\frac{z_0f'''(z_0)}{f''(z_0)}\bigg)=l
\end{equation*}
where $l\geq k\geq m\geq n\geq2$.

\end{lemma}

\noindent In 2020, Kumar and Goel \cite{goelhigher} modified the results established by Antonino and Miller \cite{antoninoandmiller}, ensuring that the Ma-Minda functions satisfy the conditions for third-order differential subordination. The modified results are mentioned below through a definition and lemma, serving as prerequisites for our main outcomes concerning the class $\mathcal{S}^*_e$.

\begin{definition}\cite{goelhigher}
Let $\Omega$ be a set in $\mathbb{C}$, $q\in Q$ and $k\geq m\geq n\geq 2$. The class of admissible operators $\xi_n[\Omega,q]$ consists of those $\xi:\mathbb{C}^4\times \mathbb{D}\rightarrow \mathbb{C}$ that satisfy the admissibility condition
\begin{equation*}
 \xi(r,s,t,u;z)\notin \Omega\quad \text{whenever}\quad z\in \mathbb{D},\quad r=q(\zeta),\quad s=m\zeta q'(\zeta),      
\end{equation*}
\begin{equation*}
 \RE\bigg(1+\frac{t}{s}\bigg)\geq m\bigg(1+\RE \frac{\zeta q''(\zeta)}{q'(\zeta)}\bigg)
\end{equation*} 
and
\begin{equation*}
\RE \frac{u}{s}\geq m^2\RE \frac{\zeta^2 q'''(\zeta)}{q'(\zeta})+3m(k-1)\RE \frac{\zeta q''(\zeta)}{q'(\zeta)} 
\end{equation*}
for $\zeta\in \partial \mathbb{D}\setminus \mathbb{E}(q)$.
\end{definition}

\begin{lemma}\cite{goelhigher}\label{8 firsttheoremthirdorder}
  Let $p\in \mathcal{H}[a,n]$ with $m\geq n\geq2$, and let $q\in Q(a)$ such that it satisfies
  \begin{equation}
       \bigg|\frac{zp'(z)}{q'(\zeta)}\bigg|\leq m,
  \end{equation}
  for $z\in \mathbb{D}$ and $\zeta\in \partial \mathbb{D}\setminus \mathbb{E}(q)$. If $\Omega$ is a set in $\mathbb{C}$, $\xi\in \Psi_n[\Omega,a]$ and
  \begin{equation*}
      \xi(p(z),zp'(z),z^2p''(z),z^3p'''(z);z)\subset \Omega,
  \end{equation*}
  then $p\prec q$.
\end{lemma}

\section{Preliminaries}

\noindent In this section, we consider the function $q(z):=e^z$ and define the admissibility class $\Psi[\Omega,q]$, where $\Omega\subset \mathbb{C}$. We know that $q$ is analytic and univalent on $\overline{\mathbb{D}}$ with $q(0)=1$ and it maps $\mathbb{D}$ onto the domain $\Delta_{e}$. Since $\mathbb{E}(q)=\phi$ for $\zeta\in \partial \mathbb{D}\setminus \mathbb{E}(q)$ if and only if $\zeta=e^{i\theta}$ $(\theta\in[0,2\pi])$. Now, consider
\begin{equation}
    |q'(\zeta)|=e^{\cos \theta}=|q(\zeta)|=:x(\theta)\label{11 1}
\end{equation}
which has maximum and minimum values of $e$ and $1/e$, respectively. It is evident that $\min |q'(\zeta)|>0$, which implies that $q\in Q(1)$, and consequently, the admissibility class $\Psi[\Omega,q]$ is well-defined. While considering $|\zeta|=1$, we observe that $q(\zeta)\in q(\partial \mathbb{D})=\partial \Delta_e = \{\varpi \in \mathbb{C}: |\log \varpi| = 1\}$. Consequently, we have $|\log q(\zeta)| = 1$ and $\log (q(\zeta)) = e^{i\theta}$ ($ \theta\in[0,2\pi]$), implies that $q(\zeta) = e^{\zeta}$. Moreover, $\zeta q'(\zeta)=e^{i\theta} e^{e^{i\theta}}$ and
\begin{equation}
    \frac{\zeta q''(\zeta)}{q'(\zeta)}=e^{i\theta}.\label{11 2}
\end{equation}
Upon comparing the real parts on both sides of \eqref{11 2}, we obtain
\begin{equation}
    \RE\bigg(\frac{\zeta q''(\zeta)}{q'(\zeta)}\bigg)=\cos \theta=:y(\theta).\label{11 3} 
\end{equation}
The function $y(\theta)$ as defined in \eqref{11 3} achieves its maximum and minimum value at $\theta = 0$ and $\pi$, given by $1$ and $-1$, respectively. Moreover, the class $\Psi[\Omega, e^z]$ is precisely defined as the class of all functions $\xi: \mathbb{C}^3 \times \mathbb{D} \rightarrow \mathbb{C}$ that satisfy the following conditions:
\begin{equation*}
    \xi(r,s,t;z)\notin \Omega,
\end{equation*}
whenever
\begin{equation}
    r=q(\zeta)=e^{e^{i\theta}};\quad s=m\zeta q'(\zeta)=me^{i\theta}e^{e^{i\theta}};\quad \RE\bigg(1+\frac{t}{s}\bigg)\geq m(1+y(\theta)), \label{11 4}
\end{equation}
where $z\in \mathbb{D}$, $\theta\in[0,2\pi]$ and $m\geq 1$. If $\xi:\mathbb{C}^2\times \mathbb{D}\rightarrow \mathbb{C}$, then the admissibility condition \eqref{11 4} becomes 
\begin{equation*}
\xi(e^{e^{i\theta}},me^{i\theta}e^{e^{i\theta}};z)\notin \Omega\quad (z\in \mathbb{D}, \theta\in[0,2\pi], m\geq 1).
\end{equation*}
Taking $q(z)=e^z$, we present below the following lemma, a special case of Miller and Mocanu \cite[Theorem 2.3b]{miller}, which is necessary to prove our main results.

\begin{lemma}\label{8 theorem1}
Let $\xi\in \Psi[\Omega, e^z]$. If $p\in \mathcal{H}[1,n]$ satisfies
\begin{equation*}
        \xi(p(z),zp'(z),z^2p''(z);z)\in \Omega,
\end{equation*}
then $p(z)\prec e^z$.
\end{lemma}

Further, suppose $q(z)=e^z$, we have
\begin{equation*}
    \zeta^2\frac{q'''(\zeta)}{q'(\zeta)}=e^{2i\theta}.
\end{equation*}
After comparing the real parts of both the sides, we get
\begin{equation}
    \RE\bigg(\zeta^2\frac{q'''(\zeta)}{q'(\zeta)}\bigg)=\cos 2\theta=:w(\theta). \label{11 5} 
\end{equation}
The maximum and minimum values of $w(\theta)$ are $1$ and $-1$, attained at $\theta=0$ and $\pi/2$, respectively. Thus, if $\xi:\mathbb{C}^4\times \mathbb{D}\rightarrow\mathbb{C}$ then $\xi\in \Psi[\Omega,e^z]$, provided $\xi$ satisfies the following conditions:
\begin{equation*}
    \xi(r,s,t,u;z)\notin \Omega\quad \text{whenever}\quad r=q(\zeta)=e^{e^{i\theta}},\quad s=m\zeta q'(\zeta)=me^{i\theta}e^{e^{i\theta}},
\end{equation*}
\begin{equation*}
   \RE\bigg(1+\frac{t}{s}\bigg)\geq m(1+y(\theta)) \quad\text{and}\quad \RE \frac{u}{s}\geq m^2 w(\theta)+3m(k-1)y(\theta),
\end{equation*}
for $z\in \mathbb{D}$, $\theta\in[0,2\pi]$ and $k\geq m\geq2$. Given the above conditions, we present below a special case of Lemma \ref{8 firsttheoremthirdorder}, which is needed for deriving our results in the last section of this article.
\begin{lemma}\label{8 firstlemmathirdorder}
Let $p\in \mathcal{H}[1,n]$ with $m\geq n\geq 2$ such that for $z\in \mathbb{D}$ and $\zeta\in \partial\mathbb{D}$, it satisfies
\begin{equation*}
   |zp'(z)e^{-\zeta}|\leq m.     
\end{equation*}
If $\Omega$ is a set in $\mathbb{C}$ and $\xi\in \Psi[\Omega,e^z]$, then
\begin{equation*}
    \xi(p(z),zp'(z),z^2p''(z),z^3p'''(z);z)\subset \Omega\implies p(z)\prec e^z.
\end{equation*}
\end{lemma}

\section{Second order differential subordination}
\noindent In this section, we discuss the following second order differential subordination implications
\begin{equation*}
    p(z)+\gamma_1 zp'(z)+\gamma_2 z^2p''(z)\prec h(z) \implies p(z)\prec e^z 
\end{equation*}
for different choices of $h(z)$, by deriving conditions on the constants $\gamma_1$ and $\gamma_2$. 

We begin with following theorem by taking $h(z)=1+\sin z$:

\begin{theorem}\label{11 sinetheorem}
Suppose $\gamma_1$, $\gamma_2>0$ and either $\gamma_1-\gamma_2-e(e+1)\geq e\sinh 1$ or $1-e(1+\gamma_1+\gamma_2)\geq \sinh 1$. If $p$ is an analytic function in $\mathbb{D}$ such that $p(0)=1$ and
\begin{equation*}
        p(z)+\gamma_1 zp'(z)+\gamma_2 z^2p''(z)\prec 1+\sin z
\end{equation*}
then $p(z)\prec e^z$.
\end{theorem}

\begin{proof}
Suppose $h(z)=1+\sin z$ for $z\in \mathbb{D}$ and $\Omega:=\{\varpi\in \mathbb{C}:|\arcsin (\varpi-1)|<1\}$. Define $\xi:\mathbb{C}^3\times \mathbb{D}\rightarrow \mathbb{C}$ as $\xi(r,s,t;z)=r+\gamma_1 s+\gamma_2 t$. For $\xi\in \Psi[\Omega, \Delta_e]$, we must have $\xi(r,s,t;z)\notin \Omega$. Through \cite[Lemma 3.3]{chosine}, we note that the smallest disk containing $\Omega$ is $\{\omega\in \mathbb{C}: |\omega-1|<\sinh 1\}$. Since,
\begin{equation*}
        |\xi(r,s,t;z)-1|=|r+\gamma_1 s+\gamma_2 t-1|.
\end{equation*}
Now, we have two different cases:\\
\noindent \underline{\textbf{Case I:}} If we express $|r+\gamma_1 s+\gamma_2 t-1|$ as
\begin{align*}
       |r+\gamma_1 s+\gamma_2 t-1|&\geq |\gamma_1 s+\gamma_2 t|-|r-1|\\
       & \geq\gamma_1 |s|\bigg|1+\frac{\gamma_2}{\gamma_1}\frac{t}{s}\bigg|-(|r|+1)\\
    &\geq \gamma_1 |s|\RE\bigg(1+\frac{\gamma_2}{\gamma_1}\frac{t}{s}\bigg)-|r|-1\\
    &\geq m \gamma_1 x(\theta)\bigg(1+\frac{\gamma_2}{\gamma_1}(y(\theta)m+m-1)\bigg)-x(\theta)-1.  
\end{align*}
Here, $x(\theta)$ and $y(\theta)$ are defined in \eqref{11 1} and \eqref{11 3}, respectively. Since $m\geq 1$, we obtain  
\begin{align*}
  |r+\gamma_1 s+\gamma_2 t-1|  &\geq \gamma_1 x(\theta)\bigg(1+\frac{\gamma_2}{\gamma_1}y(\theta)\bigg)-x(\theta)-1\\
  &\geq \frac{1}{e}(\gamma_1-\gamma_2)-e-1\\
       &\geq \sinh 1.
\end{align*}
\noindent \underline{\textbf{Case II:}} If we express $|r+\gamma_1 s+\gamma_2 t-1|$ as
\begin{align*}
|r+\gamma_1 s+\gamma_2 t-1|&\geq|r-1|-|\gamma_1 s+\gamma_2 t|\\
       & \geq 1-|r|-\gamma_1 |s|\bigg|1+\frac{\gamma_2}{\gamma_1}\frac{t}{s}\bigg|\\
    &\geq 1-|r|-\gamma_1 |s|\RE\bigg(1+\frac{\gamma_2}{\gamma_1}\frac{t}{s}\bigg)\\
    &\geq 1-x(\theta)-m \gamma_1 x(\theta)\bigg(1+\frac{\gamma_2}{\gamma_1}(y(\theta)m+m-1)\bigg).  
\end{align*}          
Here, $x(\theta)$ and $y(\theta)$ are defined in \eqref{11 1} and \eqref{11 3}, respectively. Since $m\geq 1$, we obtain
\begin{align*}
  |r+\gamma_1 s+\gamma_2 t-1|  &\geq 1-x(\theta)-\gamma_1 x(\theta)\bigg(1+\frac{\gamma_2}{\gamma_1}y(\theta)\bigg)\\ 
  &\geq 1-e-e(\gamma_1+\gamma_2)\\
       &\geq \sinh 1.
\end{align*}
Clearly, $\xi(r,s,t;z)\notin\{\omega\in \mathbb{C}: |\omega-1|<\sinh 1\}$ which is enough to conclude that $\xi(r,s,t;z)\notin \Omega$. Therefore, $\xi\in \Psi[\Omega,\Delta_e]$ and thus $p(z)\prec e^z$ through Lemma \ref{8 theorem1}.
\end{proof}

\begin{theorem}
Suppose $\gamma_1$, $\gamma_2>0$ and either $\gamma_1-\gamma_2-e(e+1)\geq e^2$ or $1-e(1+\gamma_1+\gamma_2)\geq e$. If $p$ is an analytic function in $\mathbb{D}$ such that $p(0)=1$ and
\begin{equation*}
        p(z)+\gamma_1 zp'(z)+\gamma_2 z^2p''(z)\prec 1+ze^z
\end{equation*}
then $p(z)\prec e^z$.
\end{theorem}

\begin{proof}
Suppose $h(z)=1+ze^z$ for $z\in \mathbb{D}$ and $\Omega=h(\mathbb{D})$. Let us define $\xi:\mathbb{C}^3\times \mathbb{D}\rightarrow \mathbb{C}$ as $\xi(r,s,t;z)=r+\gamma_1 s+\gamma_2 t$. For $\xi\in \Psi[\Omega, \Delta_e]$, we need to have $\xi(r,s,t;z)\notin \Omega$. Through \cite[Lemma 3.3]{kumar-ganganiaCardioid-2021},
we observe that the smallest disk containing $\Omega$ is $\{\omega\in \mathbb{C}: |\omega-1|<e\}$. Since
\begin{equation*}
        |\xi(r,s,t;z)-1|=|r+\gamma_1 s+\gamma_2 t-1|.
\end{equation*} 
We consider two different cases given as follows:\\
\noindent \underline{\textbf{Case I:}} If we express $|r+\gamma_1 s+\gamma_2 t-1|$ as
\begin{align*}
  |r+\gamma_1 s+\gamma_2 t-1|  &\geq \gamma_1 x(\theta)\bigg(1+\frac{\gamma_2}{\gamma_1}y(\theta)\bigg)-x(\theta)-1\\
  &\geq \frac{1}{e}(\gamma_1-\gamma_2)-e-1\\
       &\geq e.
\end{align*}
\noindent \underline{\textbf{Case II:}} If we express $|r+\gamma_1 s+\gamma_2 t-1|$ as
\begin{align*}
  |r+\gamma_1 s+\gamma_2 t-1|  &\geq 1-x(\theta)-\gamma_1 x(\theta)\bigg(1+\frac{\gamma_2}{\gamma_1}y(\theta)\bigg)\\
  &\geq 1-e-e(\gamma_1+\gamma_2)\\
       &\geq e.
\end{align*}
Clearly, $\xi(r,s,t;z)\notin\{\omega\in \mathbb{C}: |\omega-1|<e\}$ which suffices to conclude that $\xi(r,s,t;z)\notin \Omega$. Therefore, $\xi\in \Psi[\Omega,\Delta_e]$ and thus $p(z)\prec e^z$ by using Lemma \ref{8 theorem1}.
\end{proof}

\begin{theorem}
Suppose $\gamma_1$, $\gamma_2>0$ and either $\gamma_1-\gamma_2-e(e+1)\geq \sqrt{2}e$ or $1-e(1+\gamma_1+\gamma_2)\geq \sqrt{2}$. If $p$ is an analytic function in $\mathbb{D}$ such that $p(0)=1$ and
\begin{equation*}
        p(z)+\gamma_1 zp'(z)+\gamma_2 z^2p''(z)\prec z+\sqrt{1+z^2}
\end{equation*}
then $p(z)\prec e^z$.
\end{theorem}

\begin{proof}
Suppose $h(z)=z+\sqrt{1+z^2}$ for $z\in \mathbb{D}$ and $\Omega=\{\varpi\in \mathbb{C}:|\varpi^2-1|<2|\varpi|\}$. We define $\xi:\mathbb{C}^3\times \mathbb{D}\rightarrow \mathbb{C}$ as $\xi(r,s,t;z)=r+\gamma_1 s+\gamma_2 t$. For $\xi\in \Psi[\Omega, \Delta_e]$, we must have $\xi(r,s,t;z)\notin \Omega$. From the graph of $z+\sqrt{1+z^2}$ (see Fig. \ref{crescent}), we note that $\Omega$ is constructed by the circles
\begin{equation*}
        C_1:|z-1|=\sqrt{2}\quad\text{and}\quad C_2:|z+1|=\sqrt{2}.
\end{equation*}

\begin{figure}[h]
\centering

   \includegraphics[width=8cm, height=5cm]{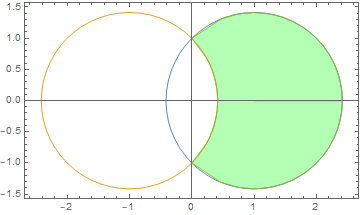}
 \caption{Graph of two circles, namely $C_1$ (blue boundary) and $C_2$ (orange boundary). While the shaded region (solid green) represents $z+\sqrt{1+z^2}$.}
  \label{crescent}
\end{figure}

\noindent It is obvious that $\Omega$ contains the disk enclosed by $C_1$ and excludes the portion of the disk enclosed by $C_2\cap C_1$. Since
\begin{equation*}
        |\xi(r,s,t;z)-1|=|r+\gamma_1 s+\gamma_2 t-1|.
\end{equation*}
Further, we have\\
\underline{\textbf{Case I:}} If we express $|r+\gamma_1 s+\gamma_2 t-1|$ as
\begin{align*}
  |r+\gamma_1 s+\gamma_2 t-1|  &\geq \gamma_1 x(\theta)\bigg(1+\frac{\gamma_2}{\gamma_1}y(\theta)\bigg)-x(\theta)-1\\
  &\geq \frac{1}{e}(\gamma_1-\gamma_2)-e-1\\
       &\geq \sqrt{2}.
\end{align*} 
\noindent \underline{\textbf{Case II:}} If we express $|r+\gamma_1 s+\gamma_2 t-1|$ as
\begin{align*}
  |r+\gamma_1 s+\gamma_2 t-1|  &\geq 1-x(\theta)-\gamma_1 x(\theta)\bigg(1+\frac{\gamma_2}{\gamma_1}y(\theta)\bigg)\\ 
  &\geq 1-e-e(\gamma_1+\gamma_2)\\
       &\geq \sqrt{2}.
\end{align*}
The fact that $\xi(r,s,t;z)\notin C_1$ suffices us to deduce that $\xi(r,s,t;z)\notin \Omega$. 
Consequently, $\xi\in \Psi[\Omega,\Delta_e]$ and thus $p(z)\prec e^z$ by using Lemma \ref{8 theorem1}.
\end{proof}

Now, we conclude this section by choosing $h(z)$ to be $1+\sinh^{-1}z$. 
\begin{theorem}
Suppose $\gamma_1$, $\gamma_2>0$ and either $2(\gamma_1-\gamma_2-e(e+1))\geq \pi e$ or $2(1-e(1+\gamma_1+\gamma_2))\geq \pi$. If $p$ is an analytic function in $\mathbb{D}$ such that $p(0)=1$ and
\begin{equation*}
        p(z)+\gamma_1 zp'(z)+\gamma_2 z^2p''(z)\prec 1+\sinh^{-1} z
\end{equation*}
then $p(z)\prec e^z$.
\end{theorem}

\begin{proof}
Suppose $h(z)=1+\sinh^{-1} z$ for $z\in \mathbb{D}$ and $\Omega=h(\mathbb{D})=\{\varpi\in\mathbb{C}:|\sinh (\varpi-1)|<1\}$. We define $\xi:\mathbb{C}^3\times \mathbb{D}\rightarrow \mathbb{C}$ as $\xi(r,s,t;z)=r+\gamma_1 s+\gamma_2 t$. For $\xi\in \Psi[\Omega, \Delta_e]$, we must have $\xi(r,s,t;z)\notin \Omega$. Through \cite[Remark 2.7]{kush}, we note that the disk $\{\omega\in\mathbb{C}: |\omega-1|<\pi/2\}$ is the smallest disk containing $\Omega$. So,
\begin{equation*}
        |\xi(r,s,t;z)-1|=|r+\gamma_1 s+\gamma_2 t-1|.
\end{equation*}
We consider two different cases given as follows:\\
\underline{\textbf{Case I:}} If we express $|r+\gamma_1 s+\gamma_2 t-1|$ as
\begin{align*}
  |r+\gamma_1 s+\gamma_2 t-1|  &\geq \gamma_1 x(\theta)\bigg(1+\frac{\gamma_2}{\gamma_1}y(\theta)\bigg)-x(\theta)-1\\
  &\geq \frac{1}{e}(\gamma_1-\gamma_2)-e-1\\
       &\geq \frac{\pi}{2}.
\end{align*}

\noindent \underline{\textbf{Case II:}} If we express $|r+\gamma_1 s+\gamma_2 t-1|$ as
\begin{align*}
  |r+\gamma_1 s+\gamma_2 t-1|  &\geq 1-x(\theta)-\gamma_1 x(\theta)\bigg(1+\frac{\gamma_2}{\gamma_1}y(\theta)\bigg)\\ 
  &\geq 1-e-e(\gamma_1+\gamma_2)\\
       &\geq \frac{\pi}{2}.
\end{align*}
Clearly, $\xi(r,s,t;z)$ does not belong to the disk $\{\omega\in \mathbb{C}: |\omega-1|<\pi/2\}$ which suffices us to conclude that $\xi(r,s,t;z)\notin \Omega$. Therefore, $\xi\in \Psi[\Omega,\Delta_e]$ and thus $p(z)\prec e^z$ by using Lemma \ref{8 theorem1}.
\end{proof}

\section{Third order differential subordination}

\noindent The objective of this section is to determine the sufficient conditions, obtained by finding the positive real numbers $\gamma_1$, $\gamma_2$ and $\gamma_3$ involving the constants $m$ and $k$ (defined in Lemma \ref{lemmaformk}) to satisfy the third order differential subordination implication, given by
\begin{equation*}
    p(z)+\gamma_1 z p'(z)+\gamma_2 z^2p''(z)+\gamma_3 z^3p'''(z)\prec h(z)\implies p(z)\prec e^z.
\end{equation*}

\begin{theorem}\label{11 thirdordersine}
Suppose $\gamma_1$, $\gamma_2$, $\gamma_3>0$ and either $\gamma_1-\gamma_2-m^2\gamma_3-3m(k-1)\gamma_3-e(e+1)\geq e\sinh 1$ or $1-e-e(\gamma_1+\gamma_2+m^2\gamma_3+3m(k-1)\gamma_3)\geq e\sinh 1$. If $p$ is an analytic function in $\mathbb{D}$ such that $p(0)=1$ and
\begin{equation*}
        p(z)+\gamma_1 zp'(z)+\gamma_2 z^2p''(z)+\gamma_3 z^3p'''(z)\prec 1+\sin z
\end{equation*}
then $p(z)\prec e^z$.
\end{theorem}

\begin{proof}
Suppose $h(z)=1+\sin z$ for $z\in \mathbb{D}$ and $\Omega=\{\varpi\in \mathbb{C}:|\arcsin (\varpi-1)|<1\}$. Let us define $\xi:\mathbb{C}^4\times \mathbb{D}\rightarrow \mathbb{C}$ as $\xi(r,s,t,u;z)=r+\gamma_1 s+\gamma_2 t+\gamma_3 u$. For $\xi\in \Psi[\Omega, \Delta_e]$, it is required that $\xi(r,s,t,u;z)\notin \Omega$. Through \cite[Lemma 3.3]{chosine}, we note that the smallest disk containing $\Omega$ is $\{\omega\in \mathbb{C}: |\omega-1|<\sinh 1\}$. Since,
\begin{equation*}
        |\xi(r,s,t,u;z)-1|=|r+\gamma_1 s+\gamma_2 t+\gamma_3 u-1|.
\end{equation*}
Now, we consider two different cases, given as follows:\\
\noindent \underline{\textbf{Case I:}} If we express $|r+\gamma_1 s+\gamma_2 t+\gamma_3 u-1|$ as
\begin{align*}
       |r+\gamma_1 s+\gamma_2 t+\gamma_3 u-1|&\geq |\gamma_1 s+\gamma_2 t+\gamma_3 u|-|r-1|\\
       &\geq \gamma_1|s|\bigg|1+\frac{\gamma_2}{\gamma_1}\frac{t}{s}+\frac{\gamma_3}{\gamma_1}\frac{u}{s}\bigg|-(|r|+1)\\
       &\geq \gamma_1 |s|\RE\bigg(1+\frac{\gamma_2}{\gamma_1}\frac{t}{s}+\frac{\gamma_3}{\gamma_1}\frac{u}{s}\bigg)-|r|-1\\
       &\geq m\gamma_1 x(\theta)\bigg(1+\frac{\gamma_2}{\gamma_1}(y(\theta)m+m-1)+\frac{\gamma_3}{\gamma_1}(m^2w(\theta)+3m(k-1)y(\theta))\bigg)\\
       &\quad-x(\theta)-1.
\end{align*}
Here, $x(\theta)$, $y(\theta)$ and $w(\theta)$ are defined in \eqref{11 1}, \eqref{11 3} and \eqref{11 5}, respectively. Since $m\geq 1$, we obtain 
\begin{align*}
     |r+\gamma_1 s+\gamma_2 t+\gamma_3 u-1|&\geq \gamma_1 x(\theta)\bigg(1+\frac{\gamma_2}{\gamma_1}y(\theta)+\frac{\gamma_3}{\gamma_1}(m^2w(\theta)+3m(k-1)y(\theta))\bigg)-x(\theta)-1\\
     &\geq \frac{1}{e}(\gamma_1-\gamma_2-\gamma_3(m^2+3m(k-1))-e-1\\
     &\geq \sinh 1.
\end{align*}

\noindent \underline{\textbf{Case II:}} If we express $|r+\gamma_1 s+\gamma_2 t+\gamma_3 u-1|$ as
\begin{align*}
       |r+\gamma_1 s+\gamma_2 t+\gamma_3 u-1|&\geq |r-1| -|\gamma_1 s+\gamma_2 t+\gamma_3 u|\\
       &\geq 1-|r|-\gamma_1|s|\bigg|1+\frac{\gamma_2}{\gamma_1}\frac{t}{s}+\frac{\gamma_3}{\gamma_1}\frac{u}{s}\bigg|\\
       &\geq 1-|r|-\gamma_1 |s|\RE\bigg(1+\frac{\gamma_2}{\gamma_1}\frac{t}{s}+\frac{\gamma_3}{\gamma_1}\frac{u}{s}\bigg)\\
       &\geq 1-x(\theta)- m\gamma_1 x(\theta)\bigg(1+\frac{\gamma_2}{\gamma_1}(y(\theta)m+m-1)+\frac{\gamma_3}{\gamma_1}(m^2w(\theta)\\
       &\quad +3m(k-1)y(\theta))\bigg).
\end{align*}
Here, $x(\theta)$, $y(\theta)$ and $w(\theta)$ are defined in \eqref{11 1}, \eqref{11 3} and \eqref{11 5}, respectively. Since $m\geq 1$, we obtain 
\begin{align*}
     |r+\gamma_1 s+\gamma_2 t+\gamma_3 u-1|&\geq 1-x(\theta)-\gamma_1 x(\theta)\bigg(1+\frac{\gamma_2}{\gamma_1}y(\theta)+\frac{\gamma_3}{\gamma_1}(m^2w(\theta)+3m(k-1)y(\theta))\bigg)\\
     &\geq 1-e-e(\gamma_1+\gamma_2+\gamma_3(m^2+3m(k-1))\\
     &\geq \sinh 1.
\end{align*}
Clearly, $\xi(r,s,t,u;z)$ does not belong to the disk $\{\omega\in \mathbb{C}: |\omega-1|<\sinh 1\}$ indicates that $\xi(r,s,t,u;z)\notin \Omega$. Therefore, $\xi\in \Psi[\Omega,\Delta_e]$ and thus $p(z)\prec e^z$ by employing Lemma \ref{8 firstlemmathirdorder}.
\end{proof}

\begin{theorem}
Suppose $\gamma_1$, $\gamma_2$, $\gamma_3>0$ and either $\gamma_1-\gamma_2-m^2\gamma_3-3m(k-1)\gamma_3-e(e+1)\geq e^2$ or $1-e-e(\gamma_1+\gamma_2+m^2\gamma_3+3m(k-1)\gamma_3)\geq e$. If $p$ is an analytic function in $\mathbb{D}$ such that $p(0)=1$ and
\begin{equation*}
        p(z)+\gamma_1 zp'(z)+\gamma_2 z^2p''(z)+\gamma_3 z^3p'''(z)\prec 1+ze^z
\end{equation*}
then $p(z)\prec e^z$.
\end{theorem}

\begin{proof}
Suppose $h(z)=1+ze^z$ for $z\in \mathbb{D}$ and $h(\mathbb{D})=\Omega$. We define $\xi:\mathbb{C}^4\times \mathbb{D}\rightarrow \mathbb{C}$ as $\xi(r,s,t,u;z)=r+\gamma_1 s+\gamma_2 t+\gamma_3 u$. For $\xi\in \Psi[\Omega, \Delta_e]$, we need to have that $\xi(r,s,t,u;z)\notin \Omega$. Through \cite[Lemma 3.3]{kumar-ganganiaCardioid-2021}, we note that the smallest disk containing $\Omega$ is $\{\omega\in\mathbb{C}: |\omega-1|<e\}$. Since
\begin{equation*}
        |\xi(r,s,t,u;z)-1|=|r+\gamma_1 s+\gamma_2 t+\gamma_3 u-1|.
\end{equation*}
Further, we consider the following cases:\\
\noindent \underline{\textbf{Case I:}} If we express $|r+\gamma_1 s+\gamma_2 t+\gamma_3 u-1|$ as
\begin{align*}
       |r+\gamma_1 s+\gamma_2 t+\gamma_3 u-1|&\geq x(\theta)\bigg(\gamma_1+\gamma_2 y(\theta)+\gamma_3(m^2w(\theta)+3m(k-1)y(\theta))\bigg)-x(\theta)-1\\
       &\geq \frac{1}{e}(\gamma_1-\gamma_2-\gamma_3(m^2+3m(k-1))-e-1\\
       &\geq e.
\end{align*}

\noindent \underline{\textbf{Case II:}} If we express $|r+\gamma_1 s+\gamma_2 t+\gamma_3 u-1|$ as
\begin{align*}
       |r+\gamma_1 s+\gamma_2 t+\gamma_3 u-1|&\geq 1-x(\theta)-x(\theta)\bigg(\gamma_1+\gamma_2 y(\theta)+\gamma_3(m^2w(\theta)+3m(k-1)y(\theta))\bigg)\\
       &\geq 1-e-e(\gamma_1+\gamma_2+\gamma_3(m^2+3m(k-1))\\
       &\geq e.
\end{align*}
Clearly, $\xi(r,s,t,u;z)$ does not belong to the disk $\{\omega\in \mathbb{C}: |\omega-1|<e\}$, is enough to conclude that $\xi(r,s,t,u;z)\notin \Omega$. Therefore, $\xi\in \Psi[\Omega,\Delta_e]$ and thus $p(z)\prec e^z$ by using Lemma \ref{8 firstlemmathirdorder}.
\end{proof}

\begin{theorem}
Suppose $\gamma_1$, $\gamma_2$, $\gamma_3>0$ and either $\gamma_1-\gamma_2-m^2\gamma_3-3m(k-1)\gamma_3-e(e+1)\geq \sqrt{2}e$ or $1-e-e(\gamma_1+\gamma_2+m^2\gamma_3+3m(k-1)\gamma_3)\geq \sqrt{2}$. If $p$ is an analytic function in $\mathbb{D}$ such that $p(0)=1$ and
\begin{equation*}
        p(z)+\gamma_1 zp'(z)+\gamma_2 z^2p''(z)+\gamma_3 z^3p'''(z)\prec z+\sqrt{1+z^2}
\end{equation*}
then $p(z)\prec e^z$.
\end{theorem}

\begin{proof}
Suppose $h(z)=z+\sqrt{1+z^2}$ for $z\in \mathbb{D}$ and $\Omega=\{\varpi\in \mathbb{C}:|\varpi^2-1|<2|\varpi|\}$. Define $\xi:\mathbb{C}^4\times \mathbb{D}\rightarrow \mathbb{C}$ as $\xi(r,s,t,u;z)=r+\gamma_1 s+\gamma_2 t+\gamma_3 u$. For $\xi\in \Psi[\Omega, \Delta_e]$, we must have $\xi(r,s,t,u;z)\notin \Omega$. From the graph of $z+\sqrt{1+z^2}$ (see Fig. \ref{crescent}), we note that $\Omega$ is constructed by two circles
\begin{equation*}
        C_1:|z-1|=\sqrt{2}\quad\text{and}\quad C_2:|z+1|=\sqrt{2}.
\end{equation*}
It is obvious that $\Omega$ contains the disk enclosed by $C_1$ and excludes the portion of the disk enclosed by $C_2\cap C_1$. Since
\begin{equation*}
        |\xi(r,s,t,u;z)-1|=|r+\gamma_1 s+\gamma_2 t+\gamma_3 u-1|.
\end{equation*}
Further, we have\\
\noindent \underline{\textbf{Case I:}} If we express $|r+\gamma_1 s+\gamma_2 t+\gamma_3 u-1|$ as
\begin{align*}
        |r+\gamma_1 s+\gamma_2 t+\gamma_3 u-1|&\geq x(\theta)\bigg(\gamma_1+\gamma_2 y(\theta)+\gamma_3(m^2w(\theta)+3m(k-1)y(\theta))\bigg)-x(\theta)-1\\
       &\geq \frac{1}{e}(\gamma_1-\gamma_2-\gamma_3(m^2+3m(k-1))-e-1\\
       &\geq \sqrt{2}.
\end{align*} 
\noindent \underline{\textbf{Case II:}} If we express $|r+\gamma_1 s+\gamma_2 t+\gamma_3 u-1|$ as
\begin{align*}
       |r+\gamma_1 s+\gamma_2 t+\gamma_3 u-1|&\geq 1-x(\theta)-x(\theta)\bigg(\gamma_1+\gamma_2 y(\theta)+\gamma_3(m^2w(\theta)+3m(k-1)y(\theta))\bigg)\\
       &\geq 1-e-e(\gamma_1+\gamma_2+\gamma_3(m^2+3m(k-1))\\
     &\geq \sqrt{2}.
       \end{align*}
Thus, we can say that $\xi(r,s,t,u;z)$ does not belong to the circle $C_1$ which suffices us to deduce that $\xi(r,s,t,u;z)\notin \Omega$. Therefore, $\xi\in \Psi[\Omega,\Delta_e]$ and thus $p(z)\prec e^z$ by using Lemma \ref{8 firstlemmathirdorder}.
\end{proof}

We conclude this section by considering $h(z)=1+\sinh^{-1}z$. 
\begin{theorem}
Suppose $\gamma_1$, $\gamma_2$, $\gamma_3>0$ and either $2(\gamma_1-\gamma_2 -m^2\gamma_3-3m(k-1)\gamma_3-e(e+1))\geq \pi e$ or $2(1-e-e(\gamma_1+\gamma_2 +m^2\gamma_3+3m(k-1)\gamma_3))\geq \pi$. If $p$ is an analytic function in $\mathbb{D}$ such that $p(0)=1$ and
\begin{equation*}
        p(z)+\gamma_1 zp'(z)+\gamma_2 z^2p''(z)+\gamma_3 z^3p'''(z)\prec 1+\sinh^{-1} z
\end{equation*}
then $p(z)\prec e^z$.
\end{theorem}

\begin{proof}
Suppose $h(z)=1+\sinh^{-1} z$ for $z\in \mathbb{D}$ and $\Omega=\{\varpi\in\mathbb{C}:|\sinh (\varpi-1)|<1\}$. Define $\xi:\mathbb{C}^4\times \mathbb{D}\rightarrow \mathbb{C}$ as $\xi(r,s,t,u;z)=r+\gamma_1 s+\gamma_2 t+\gamma_3 u$. For $\xi\in \Psi[\Omega, \Delta_e]$, it is required that $\xi(r,s,t,u;z)\notin \Omega$. Through \cite[Remark 2.7]{kush}, we note that the smallest disk containing $\Omega$ is $\{\omega\in\mathbb{C}: |\omega-1|<\pi/2\}$. Since
\begin{equation*}
        |\xi(r,s,t,u;z)-1|=|r+\gamma_1 s+\gamma_2 t+\gamma_3 u-1|.
\end{equation*}
Now, we deal with two cases, given as follows:\\
\noindent \underline{\textbf{Case I:}} If we express $|r+\gamma_1 s+\gamma_2 t+\gamma_3 u-1|$ as
\begin{align*}
      |r+\gamma_1 s+\gamma_2 t+\gamma_3 u-1|&\geq x(\theta)\bigg(\gamma_1+\gamma_2 y(\theta)+\gamma_3(m^2w(\theta)+3m(k-1)y(\theta))\bigg)-x(\theta)-1\\
       &\geq \frac{1}{e}(\gamma_1-\gamma_2-\gamma_3(m^2+3m(k-1))-e-1\\
       &\geq \frac{\pi}{2}.
\end{align*}
\noindent \underline{\textbf{Case II:}} If we express $|r+\gamma_1 s+\gamma_2 t+\gamma_3 u-1|$ as
\begin{align*}
       |r+\gamma_1 s+\gamma_2 t+\gamma_3 u-1|&\geq 1-x(\theta)-x(\theta)\bigg(\gamma_1+\gamma_2 y(\theta)+\gamma_3(m^2w(\theta)+3m(k-1)y(\theta))\bigg)\\
       &\geq 1-e-e(\gamma_1+\gamma_2+\gamma_3(m^2+3m(k-1))\\
       &\geq  \frac{\pi}{2}.
       \end{align*}
Clearly, $\xi(r,s,t,u;z)\notin\{\omega\in \mathbb{C}: |\omega-1|<\pi/2\}$ which suffices to prove that $\xi(r,s,t,u;z)\notin \Omega$. Therefore, $\xi\in \Psi[\Omega,\Delta_e]$ and thus $p(z)\prec e^z$ by Lemma \ref{8 firstlemmathirdorder}.
\end{proof}

				%
\subsection*{Acknowledgment}
Neha Verma is thankful to the Department of Applied Mathematics, Delhi Technological University, New Delhi-110042 for providing Research Fellowship.

\end{document}